\documentclass{amsart}
\usepackage{amsfonts}
\usepackage{amsmath}
\usepackage{amssymb}
\usepackage{amsthm}
\usepackage{esint}
\usepackage{yhmath}
\usepackage{enumerate}
\usepackage[noadjust]{cite}

\newcommand{\sps}[1]{\left( #1 \right)}

\newcommand{\f}[2]{\tfrac{#1}{#2}}
\newcommand{\ff}[2]{\frac{#1}{#2}}
\newcommand{\n}[1]{\lVert #1 \rVert}
\newcommand{\ns}[1]{\left\lVert #1 \right\rVert}
\renewcommand{\a}[1]{\lvert #1 \rvert}
\newcommand{\as}[1]{\left\lvert #1 \right\rvert}
\newcommand{\eval}[1]{\left. #1 \right\rvert}
\ifdefined\stretchybydefault
\newcommand{\abs}[1]{\left \lvert #1 \right \rvert}

\newcommand{\norm}[1]{\left \lVert #1 \right \rVert}

\else

\newcommand{\abs}[1]{\lvert #1 \rvert}

\newcommand{\norm}[1]{\lVert #1 \rVert}

\fi
\newcommand{\tn}[1]{{\left\vert\kern-0.25ex\left\vert\kern-0.25ex\left\vert #1 
    \right\vert\kern-0.25ex\right\vert\kern-0.25ex\right\vert}}
\let\j\jap

\newcommand{\loc}{{\mathrm{loc}}}

\DeclareMathOperator{\Lip}{Lip}

\DeclareMathOperator{\divergence}{div}
\renewcommand{\div}{\divergence}

\newcommand{\R}{\mathbb{R}}

\newcommand{\C}{\mathbb{C}}

\newcommand{\g}{\mbox{\fontfamily{pzc}\selectfont g}}

\newcommand{\mc}[1]{\mathcal{#1}}

\providecommand{\comment}[1]{}

\renewcommand{\bar}{\overline}

\renewcommand{\ll}{\lesssim}
\renewcommand{\gg}{\gtrsim}
\let\al\alpha
\let\b\beta
\let\z\zeta
\let\g\gamma

\let\d\delta
\let\e\epsilon

\let\th\theta
\let\Si\Sigma
\let\si\sigma
\let\ti\tilde
\let\ld\lambda
\let\lp\Delta
\let\na\nabla
\let\Ld\Lambda

\let\cd\cdot
\let\om\omega
\let\Om\Omega

\newcommand{\pd}{\partial}

\newcommand{\thnorm}[1]{|\kern-1pt|\kern-1pt| #1 |\kern-1pt|\kern-1pt|}
\theoremstyle{remark}
\newtheorem*{remark*}{Remark}
\theoremstyle{definition}
\newtheorem*{openproblem*}{Open Problem}
\newtheorem*{dfn*}{Definition}

\theoremstyle{plain}
\newtheorem*{theorem*}{Theorem}
\newtheorem*{proposition*}{Proposition}
\newtheorem*{conjecture*}{Conjecture}
\newtheorem{theorem}{Theorem}[section]

\newtheorem*{lemma*}{Lemma}
\newtheorem{lemma}[theorem]{Lemma}
\newtheorem*{corollary*}{Corollary}
\newtheorem{corollary}[theorem]{Corollary}

\newcommand{\qq}{\quad\quad}

\binoppenalty=\maxdimen
\relpenalty=\maxdimen

\begin{document}
\title[Uniqueness in Calder\'on's problem]{Uniqueness in Calder\'on's problem for conductivities with unbounded gradient}
\author{Boaz Haberman}
\thanks{This material is based upon work supported by the National Science Foundation Graduate Research Fellowship under Grant No. DGE 1106400. Any opinion, findings, and conclusions or recommendations expressed in this material are those of the author and do not necessarily reflect the views of the National Science Foundation.}
\begin{abstract}
We prove uniqueness in the inverse conductivity problem for uniformly elliptic conductivities in $W^{s,p}(\Omega)$, where $\Omega \subset \R^n$ is Lipschitz, $3\leq n \leq 6$, and $s$ and $p$ are such that $ W^{s,p}(\Omega)\not \subset W^{1,\infty}(\Omega)$. In particular, we obtain uniqueness for conductivities in $W^{1,n}(\Omega)$ ($n=3,4$). This improves on the result of the author and Tataru, who assumed that the conductivity is Lipschitz. 
\end{abstract}
\maketitle
\section{Introduction}
Let $\Omega\subset \R^n$ be a bounded domain with Lipschitz boundary. To a positive real-valued function $\g$ on $\Om$ with $0<c<\g<c^{-1}$ we associate an elliptic operator $L_\g$ in divergence form:
\[L_\gamma u := \div(\gamma \nabla u).\]
Given $f\in H^{1/2}(\pd \Omega)$, there exists a unique solution $u_f$ to the Dirichlet problem 
\begin{align*}
L_\gamma u_f & = 0 \quad \text{in $\Omega$}\\
\eval{u_f}_{\pd\Omega} & = f,
\end{align*}
and  we define the Dirichlet-to-Neumann map $\Lambda_\gamma: H^{1/2}(\pd \Om) \to H^{-1/2}(\pd \Om)$ formally by
\[\Lambda_\gamma(f) := \eval{\gamma \frac{\pd u_f}{\pd \nu}}_{\pd \Omega},\]
where $\pd/\pd_\nu$ is the outward normal derivative at the boundary. By identifying $H^{1/2}(\pd\Om)$ with the quotient $H^1(\Om)/H^1_0(\Om)$, we can interpret this definition in a weak sense as follows: If $v \in H^1(\Om)$ satisfies $\eval{v}_{\pd \Om} = g$, then
\[\j{\Ld_\g(f), g} := \int_\Om \g\na u_f\cd \bar{\na v}\,dx,\]
where the notation $\j{\cd,\cd}$ indicates the pairing between $H^{-1/2}(\pd \Om)$ and $H^{1/2}(\pd \Om)$. Our main theorem is that the map $\g \mapsto \Ld_\g$ is injective for certain $\g$:
\begin{theorem}
\label{maintheorem}
The map $\g \mapsto \Ld_\g$ is injective for $\g\in W^{s,p}(\Om)$, where
\begin{align*}
(s,p)&=\begin{cases}
(1,n) & n=3,4\\
(1 + (1-\th)(\ff 1 2 - \ff 2 n), \ff{n}{1-\th}) & n=5,6
\end{cases}
\end{align*}
and $\th \in [0,1)$.
\end{theorem} 

This problem was introduced by Calder\'on, who proved uniqueness in~\cite{Calderon1980} for the linearized problem. The basic approach to this problem in this paper is the method introduced by Sylvester and Uhlmann in~\cite{Sylvester1987}, where they proved uniqueness for $n\geq 3$ and $\g \in C^2$ based on ideas in~\cite{Calderon1980}. It is of interest to determine how much this regularity condition can be relaxed. Uniqueness is known to fail (at least in the anistropic problem) for conductivities which are sufficiently singular, as was shown in~\cite{greenleaf2003,Kohn2008}. 

For $n\geq 3$, the regularity assumption in~\cite{Sylvester1987} was relaxed to $\g \in C^{3/2+}$ by Brown~\cite{Brown1996}, to $C^{3/2}$ by~\cite{Paivarinta2003}, to $W^{3/2,2n+}$ in~\cite{Brown2003}, and to $C^1$ conductivities or Lipschitz conductivities close to the identity in~\cite{Haberman2013}. Recently, the smallness condition for Lipschitz conductivities was removed in~\cite{Caro2014}. In~\cite{Nguyen2014}, uniqueness was shown in three dimensions for conductivities in $W^{3/2+,2}$.

In two dimensions, the low-regularity theory is fairly well-understood. There are essentially sharp results, even for anisotropic conductivities. In particular, uniqueness holds for $\g$ in $L^\infty$, which is is invariant under the scaling associated to the Dirichlet problem. The methods in the plane are somewhat different, and we refer the reader to~\cite{Astala2011} and references therein. 

There are some reasons to doubt that uniqueness holds in higher dimensions for conductivities with less than one derivative. Calder\'on's problem seems to be closely related to unique continuation; in particular, most of the progress in both of these problems involves Carleman estimates. Unique continuation in the plane holds for elliptic operators in divergence form when the coefficients are merely bounded~\cite{Alessandrini1992}; in higher dimensions, however, this is only known for Lipschitz coefficients~\cite{Aronszajn1962}. Furthermore, there are counterexamples to unique continuation for elliptic equations where the coefficients are $C^\al$ with any $\al < 1$~\cite{Plis1963, Miller1973, Mandache1998}. 

The conductivity equation $\div (\g \na u) =0$ is equivalent to 
\[(\lp + A\cd \na) u = 0,\]
where $A = \na \log \g$. Unique continuation holds for this equation as long as $A \in L^n$~\cite{Wolff1992,Koch2001}. Brown~\cite{Brown2003} conjectured uniqueness in the inverse conductivity problem for $\g \in W^{1,n}$. We verify that this conjecture holds in dimensions three and four.

One can also study the closely-related problem of determining a Schr\"odinger potential $q$ from the Cauchy data associated with the operator $-\lp + q$. In this setting Lavine and Nachman used the $L^p$ Carleman estimates of~\cite{Kenig1987} to show that the Cauchy data determine $q\in L^{n/2}$ (see also~\cite{Chanillo1990,DosSantosFerreira2009,Nguyen2014} for similar results). These $L^p$ Carleman estimates are the starting point for our analysis.

It was shown in~\cite{Sylvester1987} that the inverse conductivity problem reduces to the inverse problem for $-\lp + q$, where $q = \g^{-1/2} \lp \g^{1/2}$. One step in this reduction is to show that the map $\Ld_\g$ determines $\g$ and its normal derivatives at the boundary. In~\cite{Kohn1984}, Kohn and Vogelius established that for smooth conductivities, the map $\gamma \to \Lambda_\gamma$ determines the values of $\gamma$ and all of its derivatives on $\pd \Omega$. This boundary determination result holds in much greater generality~\cite{Alessandrini1990,Sylvester1988}. In particular, Brown in~\cite{Brown2013} showed that the boundary values of a $W^{1,1}$ conductivity are determined by the Dirichlet-to-Neumann map. This improvement will be a crucial ingredient in this paper.

The key idea in~\cite{Sylvester1987} is that if $\g_i$ are such that $\Ld_{\g_1}=\Ld_{\g_2}$, then
\[\int_{\Om} (q_1-q_2)\,u_1\,u_2\,dx = 0,\]
where the $u_i$ are arbitrary solutions to the Schr\"odinger equation $(-\lp + q)u_i = 0$ in $\Om$. It follows that one way to show that the potentials $q_1$ and $q_2$ coincide is to produce enough solutions to the corresponding Schr\"odinger equations that their products are dense in some sense. This idea goes back to the original paper of Calder\'on~\cite{Calderon1980}. In~\cite{Sylvester1987}, Sylvester and Uhlmann proved a uniqueness result for $C^2$ conductivities by constructing complex geometrical optics (CGO) solutions of the form $u_i = e^{x\cdot \zeta_i}(1 + \psi_i)$. Here the $\zeta_i\in \C^n$ are chosen so that $\zeta_i \cdot \zeta_i = 0$, so that $e^{x\cdot \zeta_i}$ is harmonic, and $e^{x\cdot \zeta_1} e^{x\cdot \zeta_2} = e^{ix\cdot k}$ for some fixed frequency $k\in \R^n$. In three or more dimensions, these conditions allow for an infinite family of pairs $\zeta_1,\zeta_2$ with $\abs{\zeta_i} \to \infty$. The remainders $\psi_i$ decay to zero in a suitable sense sense as $\abs{\zeta_i} \to \infty$, so that the product $u_1 u_2$ converges to $e^{ix\cdot k}$. Since $k$ is arbitrary, uniqueness follows from Fourier inversion.

To construct these CGO solutions, fix $\zeta \in \C^n$ such that $\zeta \cdot \zeta = 0$, and note that $e^{-x\cdot \zeta} \Delta( e^{x\cdot \zeta}\psi) = (\Delta + 2\zeta \cdot \nabla) \psi$. Then $u = e^{x\cdot \zeta}(1+\psi)$ solves $\Delta u = qu$ if
\begin{equation}
\label{cgoequation}
\Delta_\zeta \psi := \Delta \psi + 2 \zeta \cdot \nabla \psi = q (1+\psi).
\end{equation}
Let $m_q$ be the map sending $\psi$ to $q\psi$. We will treat this equation perturbatively, by viewing $\Delta_\zeta - m_q$ as a perturbation of $\Delta_\zeta$. The operator $\Delta_\zeta$ has a right inverse defined by 
\[\widehat{\Delta_\zeta^{-1} f}(\xi) = p_\zeta(\xi)^{-1} \hat f(\xi),\]
where
\[p_\zeta(\xi) = -\abs{\xi}^2 + 2i\zeta \cdot \xi.\]
We take  $\z$ of the form $\tau (e_1 - i e_2)$, where $e_1,e_2$ are orthogonal unit vectors. Then $\lp_\z$ is characteristic on a codimension-2 sphere $\Si_\z$, with
\[\Si_\z = \{\xi: \xi\cd e_1 = 0, \a{\xi-\tau e_2}=\tau\}.\]

To construct a solution to~\eqref{cgoequation}, we show that $m_q$ is a perturbation of $\lp_\z$ with respect to an appropriate norm. 

It was observed in~\cite{Paivarinta2003} that it is possible to construct CGO solutions for $C^1$ conductivities using Picard iteration in Sobolev spaces. However, because the inhomogeneity in~\eqref{cgoequation} is not bounded in the correct space, one cannot control these solutions.

A simplified explanation of the problem is as follows: in problems involving Carleman estimates, it is most natural to work with Sobolev spaces depending on a large parameter $\tau$ (or a small parameter in the semiclassical notation). In our problem $\tau$ is proportional to $\a{\z}$. The quantity $\tau$ is thought of as equivalent to a derivative $\na$. In view of this correspondence, we define the Sobolev space $H^s_\tau(\R^n)$ by 
\[\n{u}_{H^s_\tau} := \n{(-\lp + \tau^2)^{s/2} u}_{L^2}.\]
For the operator $\lp_\zeta^{-1}$ the following Carleman estimate holds~\cite{Salo2009}:
\[\n{u}_{H^{1}_\tau} \ll \tau \n{\lp_\z u}_{H^{-1}_\tau},\]
where $u$ is supported in some fixed compact set. This means that once we account for the physical space localization in the problem, the operator $\lp_\z^{-1}$ (heuristically speaking) maps $H^{-1}_\tau$ to $H^{1}_\tau$ with constant $\tau$. On the other hand, we have $q = \lp \f 1 2 \log \g + l.o.t.$
\begin{align}
\a{\j{m_q u, v}_{L^2}} &\ll \n{\na (\log \g) \na(u \bar v)}_{L^1}+l.o.t. \\
&\ll \n{\g}_{W^{1,\infty}} (\n{\na u}_{L^2} \n{v}_{L^2} + \n{u}_{L^2} \n{\na v}_{L^2}) + \dotsb  \\
&\ll \tau^{-1} \n{\g}_{W^{1,\infty}} \n{u}_{H^1_\tau}\n{v}_{H^1_\tau}.
\label{sobolev}
\end{align}
By duality, this implies that $m_q$ maps $H^{1}_\tau$ back to $H^{-1}_\tau$ with constant $\tau^{-1}$. This means that the composition $m_q \lp_\zeta^{-1}$ is bounded, and there is some hope of construction CGO solutions perturbatively (see~\cite{Krupchyk2014} for this type of analysis). We are trying to solve $(\lp_\zeta - m_q) u = q$, and we certainly have $q \in H^{-1}_\tau$. However, because we lost a factor of $\tau$ in the Carleman estimate, our solution only satisfies an estimate of the form $\n{\psi}_{H^1_\tau} \leq \tau \n{q}_{H^{-1}_\tau}$. This is bad, because $\psi$ is supposed to be small in $H^1_\tau$.

In~\cite{Nguyen2014}, this problem with Sobolev spaces is circumvented by showing that the first iterate in the solution procedure is bounded on average. This avoids the use of specialized spaces at the expense of requiring slightly more differentiability.

In~\cite{Haberman2013}, Tataru and the author dealt with this problem using specialised function spaces. These are inspired by the $X^{s,b}$ spaces of Bourgain~\cite{Bourgain1993}, and were used in the context of Carleman estimates by~\cite{Tataru1996a}. Define the space $\dot X^b_\zeta$ by the norm
\[\norm{u}_{\dot X^b_\zeta} = \norm{\abs{p_\zeta(\xi)}^{b} \hat u(\xi)}_{L^2},\]
where $p_\zeta(\xi) = -\abs{\xi}^2 + 2i\zeta \cdot \xi$ is the symbol of $\Delta_\zeta$. In this paper, we take $b=\pm 1/2$. It is easy to see that $\norm{\Delta^{-1}_\zeta}_{\dot X_\zeta^{-1/2} \to \dot X_\zeta^{1/2}} =1$. We will also make use of the inhomogeneous spaces $X_\zeta^{b}$ with norm
\[\norm{u}_{X_\zeta^{b}} = \norm{(\abs{\zeta} + \abs{p_\zeta(\xi)})^{b} \hat u(\xi)}_{L^2}.\]
The map $m_q$ satisfies 
\begin{equation}
\label{lipschitzbound}
\n{m_q}_{\dot X^{1/2}_\z \to \dot X^{-1/2}_\z} \ll \n{\g}_{\Lip},
\end{equation}
and we may solve~\eqref{cgoequation} perturbatively as before. This bound follows from~\eqref{sobolev}, the easy estimate
\begin{align}
\label{easyenergy}
\n{u}_{H^1_\tau} &\ll \tau^{1/2} \n{u}_{X^{1/2}_\z},
\end{align}
and the fact that for localized $u$, the $X^{1/2}_\z$ and $\dot X^{1/2}_\z$ norms are equivalent.

Solving in this way gives a CGO solution $\psi$ with $\n{\psi}_{\dot X^{1/2}_\z} \ll \n{q}_{\dot X^{-1/2}_\z}$. Unfortunately, the best bound on $\n{q}_{\dot X^{-1/2}_\z}$ is $\tau^{1/2} \n{\g}_{H^1}$. This means that the $\dot X^{1/2}_\z$ norm of CGO solutions might grow like $\tau^{1/2}$ as $\tau \to \infty$. 

By an averaging argument, however, $\n{q}_{\dot X_\z^{-1/2}}$ is bounded for a large set of $\z$ (which may depend on $q$) as long as $\g \in H^1$. This is because the $\dot X^{-1/2}$ norm is only large when $\hat q$ concentrates near $\Si_\z$. As we vary $\z$, the characteristic set $\Si_\z$ varies through a family of growing codimension 2 spheres, and $\hat q$ cannot concentrate near all of them. In particular, the estimate $\n{q}_{\dot X^{-1/2}_\z} \ll \n{\g}_{H^1}$ holds on average. Once this is established, uniqueness for $\g \in C^1$ follows from the standard arguments.

If $\na \g$ is unbounded, then we need to replace the $H^1_\tau$ norm on the left hand side of~\eqref{easyenergy} with an $L^p$ norm, where $p > 2$. We can obtain such an estimate using the methods of~\cite{Kenig1987}, which essentially give~\footnote{The author would like to thank Russell Brown for pointing this out.}
\begin{equation}
\label{halfkenig}
\n{u}_{2n/(n-2)} \ll \n{u}_{\dot X_\z^{1/2}}.
\end{equation}
This puts $u$ in a better $L^p$ space, but at the cost of a factor of $\tau$. Since there is no such room in~\eqref{lipschitzbound}, it seems that such a bound does not hold for $\g \in W^{1,p}$ if $p < \infty$. 

To do better, we need a refined version of~\eqref{halfkenig}. If $\hat u_\mu$ is supported in the region $\{\xi: d(\xi,\Si_\z) \sim \mu\}$, then we can replace~\eqref{halfkenig} by
\begin{equation}
\label{refinedhalfkenig}
\n{u_\mu}_{2n/(n-2)} \ll (\mu/\tau)^{1/n} \n{u_\mu}_{\dot X^{1/2}_\z}.
\end{equation}
and~\eqref{easyenergy} by
\begin{equation}
\label{refinedeasyenergy}
\n{u_\mu}_2 \ll (\mu\tau)^{-1/2} \n{u_\mu}_{X_\z^{1/2}}.
\end{equation}
Assume we are given $v_\nu$ satisfying a similar condition, with $\mu \leq \nu$. Define $f = \na \log \g$. Then
\begin{align*}
\a{\int (\na f)  u_\mu \bar v_\nu}& \ll  \n{\na f}_n (\mu/\tau)^{1/n} (\nu \tau)^{-1/2} \n{u_\mu}_{X^{1/2}_\z} \n{v_\nu}_{X^{1/2}_\z}.
\end{align*}
Now we exploit the fact that the Fourier transform of $u_\mu v_\nu$ is supported in $\{\xi: \a{\xi} \ll \tau, \a{\xi\cd e_1} \ll \nu\}$. By orthogonality, we may restrict $f$ to this region, so that the above becomes
\[\a{\int( \na f)  u_\mu \bar v_\nu} \ll \n{D^{1/2-1/n} D_1^{1/n-1/2} f}_n \n{u_\mu}_{X^{1/2}_\z} \n{v_\nu}_{X_\z^{1/2}}.\] 
where $D$ and $D_1$ are operators with symbols $\a{\xi}$ and $\a{\xi\cd e_1}$, respectively. An argument along these lines gives an estimate of the form
\[\n{m_{\na f}}_{X_\z^{1/2} \to X_\z^{-1/2}} \ll \n{D^{1/2-1/n} D_1^{1/n-1/2} f}_n.\]
Although we have lost $1/2-1/n$ derivatives in this estimate, this is counterbalanced by a gain of $1/2-1/n$ derivatives in the $e_1$ direction. This gain is useless if the Fourier support of $f$ concentrates near the plane perpendicular to $e_1$. However, we expect that this behavior does not occur on average, and we can take advantage of this by exploiting our freedom in choosing $\z$. 

It is easiest to average over all choices of $e_1\in S^{n-1}$. In $L^2$ we have 
\[\int_{e_1 \in S^{n-1}}\n{D^\al D_1^{-\al} f}_2^2 \,d\si(e_1) \ll \n{f}_2.\]
for $\al < 1/2$. Heuristically, we can interpolate this with the trivial observation that $\sup_{e_1 \in S^{n-1}} \n{f}_\infty \ll \n{f}_\infty$ to obtain 
\[\sps{\int_{e_1 \in S^{n-1}} \n{D^\b D_1^{-\b} f}_p^p \,d\si(e_1)}^{1/p} \ll \n{f}_p,\]
when $\b < 1/p$. In three dimensions, we have $1/2-1/3 < 1/3$, and we find that $\n{m_{\na f}}_{X^{1/2}_\z\to X^{-1/2}_\z}$ is bounded on average. In four dimensions, we have $1/2-1/4=1/4$. This causes a logarithmic divergence, which turns out to be harmless. For $n\geq 5$, however, we do not have a way to avoid losing derivatives.

The averaging argument in this paper is somewhat different from the argument in~\cite{Haberman2013}. There, $k$ was taken to be fixed, and $\hat q(k)$ was determined by testing against CGO solutions with $\z \sim \tau(\eta_1-i\eta_2)$, where $\eta_1$ and $\eta_2$ are perpendicular to the frequency $k$. This approach does not give control of $D_1^\b D_1^{-\b} f$, since averaging only over $\eta_1$ perpendicular to $k$ is useless if $\hat f$ is concentrated along the $k$ direction. 

Instead of fixing $k$, we vary the triple $\{k, \eta_1, \eta_2\}$ over an small open set of triples of orthonormal vectors. This set is essentially parameterized by $\{U e_1, Ue_2, Ue_3\}$, where the $e_i$ denote fixed orthonormal vectors, and $U$ is an orthogonal transformation. The idea is then to average the relevant quantities (which depend on $U$) with respect to the Haar measure on $O(n)$. Using all of the degrees of freedom in this way allows for an improvement in the estimate for $m_q$ and clarifies the estimate for $\n{q}_{X^{-1/2}_\z}$. This idea comes from~\cite{Nguyen2014}, where uniqueness is established in three dimensions for conductivities in $W^{3/2+,2}$. Remarkably, they showed that under this assumption the boundedness of $m_q$ on average can be proven without taking advantage of the curvature of $\Si_\z$. Using our framework, this corresponds (roughly speaking) to applying Bernstein's inequality to $u_\mu$ in the $e_1$ direction and Sobolev embedding to $f$ in the other directions, to obtain 
\begin{align*}
\a{\int (\na f) u_\mu \bar v_\nu} & \ll \n{\na f}_{L^2_{e_1} L^\infty_{e_2,e_3}} \n{u_\mu}_{L^\infty_{e_1} L^2_{e_2,e_3}} \n{v_\nu}_{L^2}\\
&\ll \n{\j D^{2+} f}_2\, \mu^{1/2}\n{u_\mu}_2 \n{v_\nu}_2\\
&\ll \nu^{-1/2} \tau^{-1} \n{D^{2+} f}_2 \n{u_\mu}_{X^{1/2}_\z} \n{u_\nu}_{X^{1/2}_\z}.
\end{align*}
Taking $f$ supported in $\{\a{\xi} \ll \tau, \a{\xi_1} \ll \nu\}$, we obtain
\[\a{\int( \na f ) u_\mu \bar v_\nu} \ll \n{D^{1/2}D_1^{-1/2} D^{1/2+} f}_2 \n{u_\mu}_{X^{1/2}_\z} \n{v_\nu}_{X^{1/2}_\z}, \]
and we can estimate $\n{D^{1/2}D_1^{-1/2} D^{1/2+} f}_2$ on average as before.

When $n \geq 7$, the situation is essentially the same, but there is a new technical difficulty. Since our methods are global in space, we need to extend the conductivities $\g_i \in W^{s,p}(\Om)$ to some $\g_i \in W^{s,p}(\R^n)$ which agree outside of $\Om$. When $s\leq 1+ 1/p$, we can do this as long as $\g_1=\g_2$ on the boundary. However, when $s > 1+1/p$, we also need $\pd_\nu \g_1 = \pd_\nu \g_2$ on the boundary, and there does not seem to be such a boundary identification result in the literature.

It is possible that one can relax the uniform ellipticity condition on $\g$. The natural condition to impose is then $\na \log \g \in L^n$, in which case $\log \g$ is only in BMO, which would correspond to the results of~\cite{Astala2011} in the plane. We will not address this issue, as it introduces numerous technical difficulties. We note, however, that the assumption in Theorem~\ref{estimatetheorem} is of this type.

\section{Notation}
Let $\zeta = \tau(e_1 - ie_2)$, where $e_1,e_2 \in \R^n$ are orthogonal unit vectors. Define the conjugated Laplacian
\[\lp_\z := e^{-x\cd \z} \lp e^{x\cd \z},\]
a differential operator whose symbol is
\[p_\z(\xi) := -\a \xi^2 + 2 i \z \cd \xi.\]
This symbol vanishes simply on the characteristic set
\[\Si_\zeta := \{\xi:  p_\zeta(\xi)=0\} = \{\xi: \xi_1 = 0, \a{\xi-\tau e_2} = \tau\},\]
which is a sphere of codimension two. In fact, it is not hard to check that
\begin{equation}
\label{pbehavior}
\a{p_\z(\xi)} \sim \begin{cases}
\tau d(\xi,\Si_\z) & d(\xi,\Si_\z) \leq \tau/8\\
\tau^2 + \a{\xi}^2 & d(\xi,\Si_\z) > \tau/8
\end{cases}
\end{equation}
where $d(\xi,\Si_\z) \sim \a{\xi_1} + \a{\a{\xi - \tau e_2}-\tau}$ is the distance from $\xi$ to $\Si_\z$. We will refer to this distance as the {\em modulation}.

Define the Banach spaces $\dot X^{b}_\z$ and $X^b_\z$ with norms
\begin{align*}
\n{u}_{\dot X^b_\z} & = \n{\a{p_\z(\xi)}^b \hat u}_{L^2} \\
\n{u}_{X^b_\z} & = \n{(\a{p_\z(\xi)}+\tau)^b \hat u}_{L^2} .
\end{align*}
We will use the Greek letters $\ld,\mu,\nu$ to represent dyadic integers of the form $2^{k}$, where $k \geq 0$. For $\ld > 1$ we define $E_\ld$ to be the set of $\xi$ with modulation comparable to $\ld$:
\[E_\ld := \{\xi: d(\xi,\Si_\z) \in (\ld/2,\ld]\}.\]
Similarly, for any $\ld$ we write
\[E_{\leq \ld} := \{\xi: d(\xi,\Si_\z) \leq \ld\}.\]
Since our problem is localized to a fixed compact set, the uncertainty principle implies that we need not distinguish frequencies which are separated on the unit scale. Therefore, by abuse of notation we will define $E_1 := E_{\leq 1}$. Let $m_\ld$ denote the characteristic function of $E_\ld$, and similarly for $m_{\leq \ld}$. 

Let $Q_\ld, Q_{\leq \ld}$ be the Fourier multipliers with symbols $m_\ld, m_{\leq \ld}$. We will wish to distinguish the cases $\ld \leq \tau/8$ and $\ld \gg \tau/8$, so we define projections onto low and high modulation by
\begin{align*}
Q_l & = \sum_{1\leq \ld \leq \tau/8} Q_\ld \\
Q_h & = \sum_{\ld > \tau/8} Q_\ld,
\end{align*}
where the $\ld$ vary over dyadic integers. By~\eqref{pbehavior}, we have
\begin{align}
\label{hfestimate}
\norm{Q_h u}_{H^1_\tau} \ll \norm{u}_{\dot X^{1/2}_\zeta}
\end{align}
We will also need the standard Littlewood-Paley projections. For these we choose a smooth dyadic partition of unity, i.e. a function $\chi \in C_0^\infty(\R)$ supported on $[1/2,2]$ such that 
\[1 = \sum_{k=-\infty}^\infty \chi(2^{-k} \rho).\]
for any $\rho > 0$. For a dyadic integer $\ld> 1$, we set $\chi_\ld(\xi) = \chi(\a{\xi}/\ld)$, and by abuse of notation we again set $\chi_1 = \sum_{\ld \leq 1} \chi_\ld$. The Littlewood-Paley projections $P_\ld$ are defined as the Fourier multipliers with symbols $\chi_\ld$. Given a direction $\om \in S^{n-1}$, we can also define the Littlewood-Paley projections $P^\om_\ld$ in the $\om$ direction using the Fourier multipliers $\chi(\a{\xi \cd \om}/\ld)$.
\section{Strichartz estimates}
Our goal in this section is to prove $L^p$ estimates for functions in $\dot X^{1/2}_\z$. We follow~\cite{Kenig1987}. Since the symbol $p_\z(\xi)$ is characteristic on a sphere $\Si_\z$, we begin with the Stein-Tomas restriction theorem.
\begin{theorem}[\cite{Stein1993,Tomas1975}]
Suppose $p \geq (2d+2)/(d-1)$. Let $\si$ denote the surface measure on $S^{d-1}$. Then
\[\n{\widehat{f \,d\si}}_{L^{p}(\R^d)} \ll \n{f}_{L^2(S^{d-1})}.\]
\end{theorem}

Let $\tau S^{d-1}$ denote the sphere of radius $\tau$. Given a set $E$ we define its $\ld$-neighborhood
\[N_\ld(E) := \{\xi: d(\xi, E) \leq \ld\}.\]
We use the following rescaled and localized variant of the restriction theorem:
\begin{corollary}
\label{localrestriction}
Let $p$ be as above. Suppose that $\hat g$ is supported in $N_\ld(\tau S^{d-1})$, where $\ld \leq \tau/8$. Then
\[\n{g}_{L^{p}} \ll \ld^{1/2} \tau^{(d-1)/2-d/p}  \n{\hat g}_{L^2(N_\ld(\tau S^{d-1}))}\]
\end{corollary}
\begin{proof}
By Fourier inversion, we have
\begin{align*}
g(x)& =c_d\int \hat g(\xi) e^{ix\cd \xi}\,d\xi\\
&= c_d \int_{\tau-\ld}^{\tau+\ld} \int_{S^{d-1}} \hat g(\rho \om) e^{i\rho \j{x,\om}}\,\rho^{d-1}\,d\rho\,d\si\\
&= c_d\int_{\tau-\ld}^{\tau+\ld} \rho^{d-1} (\hat g(\rho \om) \,d\si)^\vee (\rho x)\,d\rho
\end{align*}
By Minkowski's inequality, the restriction theorem, Cauchy-Schwarz, and Plancherel this implies that
\begin{align*}
\n{g}_{L^{p}} & \ll \int_{\tau-\ld}^{\tau+\ld} \n{(\hat g(\rho \om)\,d\si)^\vee(\rho x)}_{L^{p}(\R^d)} \,\rho^{d-1} \,d\rho\\
&\ll \int_{\tau-\ld}^{\tau+\ld} \rho^{-d/p} \n{(\hat g(\rho \om)\,d\si)^\vee(x)}_{L^{p}(\R^d)} \,\rho^{d-1} \,d\rho\\
\\
&\ll \tau^{-d/p}\int_{\tau-\ld}^{\tau+\ld} \n{\hat g(\rho \om)}_{L^2(S^{d-1})}\,\rho^{d-1}\,d\rho\\
&\ll \tau^{-d/p} (\ld \tau^{d-1})^{1/2} \sps{\int_{\tau-\ld}^{\tau+\ld} \n{\hat g(\rho\om)}_{L^2(S^{d-1})}^2 \,\rho^{d-1} \,d\rho}^{1/2}\\
&=\ld^{1/2} \tau^{(d-1)/2 - d/p} \n{\hat g}_{L^2(N_\ld(\tau S^{d-1}))}.
\end{align*}
\end{proof}
We deduce the following Strichartz-type estimates
\begin{lemma}
Let $p=2n/(n-2)$, $\ld\leq \tau/8$. Then\footnote{Strictly speaking, the $X^{1/2}_\z$ norm should be replaced with $\dot X^{1/2}_\z$, but this will not be important.}
\begin{align}
\label{bandstrichartz}
\n{Q_\ld f}_{p} &\ll (\ld/\tau)^{1/n} \n{f}_{X^{1/2}_\zeta}.\\
\label{strichartz}
\n{f}_p& \ll \n{f}_{X_\zeta^{1/2}}.
\end{align}
\end{lemma}
\begin{proof}
By a change of coordinates, we may assume $e_1 = (1,0,\dotsc, 0)$. We use the notation $\xi = (\xi_1,\xi')$.

For~\eqref{bandstrichartz}, write $g = Q_\ld f$. Note that
\[E_\ld \subset \{\xi: \a{\xi_1} \leq c\ld, \a{\a{\xi' - \tau e_2'}-\tau} \leq c\ld\}.\]
We can write $g = \phi_\ld *_{x_1} g$, where $\phi_\ld(x_1) = \ld \phi(\ld x_1)$ for some Schwartz $\phi$ and the convolution is taken in the $x_1$ variable only. By Minkowski's inequality and Young's inequality, we have
\begin{align*}
\n{g}_{p} & = \ns{\int \phi_\ld(x_1-y_1) g(y_1,x')\,dy_1}_{p} \\
&\ll \ns{\int \a{ \phi_\ld (x_1-y_1) } \n{g(y_1)}_{L^{p}_{x'}}\,dy_1}_{L^{p}_{x_1}}\\
&\ll \ld^{1/2-1/p} \n{g}_{L^2_{x_1} L^p_{x'}}.
\end{align*}
If we regard $g$ as a function in the $x'$ variable, we see that its Fourier transform lies in $N_{c\ld}(\tau S^{n-2} + \tau e_2')$. By Corollary~\ref{localrestriction} and translation invariance, we have $\n{g(x_1)}_{L^p_{x'}} \ll \ld^{1/2} \tau^{(n-2)/(2n)}\n{\hat g(x_1)}_{L^2_{x'}}$ for each $x_1$. It follows that
\[\n{g}_p \ll \ld^{1/2} \ld^{1/2-1/p} \tau^{1/2-1/n} \n{\hat g}_{2} \ll \ld^{1/n} \tau^{-1/n} \n{f}_{X^{1/2}_\zeta}. \]

For~\eqref{strichartz}, we apply~\eqref{bandstrichartz} near $\Si_\z$ and Sobolev embedding away from $\Si_\z$. On $E$ we have
\begin{align*}
\n{Q_{l} f}_p & \ll \sum_{1 \leq \ld \leq \tau/8} \n{Q_\ld f}_p \\
&\ll \sum_{1 \leq \ld \leq \tau/8} (\ld/\tau)^{1/n}\n{Q_\ld f}_{X^{1/2}_\z}\\
&\ll \sps{\sum_{1\leq \ld \leq \tau/8} \n{Q_\ld f}_{X_\z^{1/2}}^2}^{1/2} \\
&\leq \n{f}_{X_\z^{1/2}}
\end{align*}
Away from $E$ we have 
\[\n{Q_{h}f}_p \ll \n{Q_{h}f}_{H^1} \ll \n{f}_{X^{1/2}_\z}.\]
by~\eqref{hfestimate}. Combining these estimates gives the claimed inequality.
\end{proof}

\section{Bilinear estimates}
Given a tempered distribution $f\in \mc S'(\R^n)$, define the map $m_f: \mc S(\R^n) \to \mc S'(\R^n)$ by $m_f \phi := f \phi$. We would like to control $\n{m_f}_{X^{1/2}_\z\to X^{-1/2}_\z}$. By duality, this is equivalent to establishing a bilinear estimate of the form
\[\a{m_f(u,v)} \ll \n{u}_{X^{1/2}_\z} \n{v}_{X^{1/2}_\z},\]
where
\[m_f(u,v) = \j{m_f u, v}.\]
Suppose that $f \in L^{n/2}$. By~\eqref{strichartz}, we have
\begin{align}
\a{m_f(u,v)} & \ll \n{f}_{n/2} \n{u}_{2n/(n-2)} \n{v}_{2n/(n-2)}\\
&\ll \n{f}_{n/2} \n{u}_{X^{1/2}_\z} \n{v}_{X^{1/2}_\z}.
\label{lowerorder}
\end{align}
We also have
\begin{equation}
\a{m_f(u,v)} \ll \n{f}_\infty \n{u}_2 \n{v}_2 \ll \tau^{-1} \n{f}_\infty \n{u}_{X^{1/2}_\z} \n{v}_{X^{1/2}_\z}.
\label{linfinity}
\end{equation}
A more difficult task is to control $m_{\na f}$. We record the main computation in the following lemma:
\begin{lemma}
\label{calculation}
Let $s,p,\th$ be as in Theorem~\ref{maintheorem}. Let $1/q = 1/2-1/p$. There is some $\al > 0$ such for fixed $\ld \leq 100\tau$, we have
\begin{align}
\label{leq}
\sum_{\substack{\mu\leq \nu\leq \tau/8\\\nu <\ld}}  (\nu/\ld)^{(1-\th)/n}  \n{Q_\mu u}_q\n{Q_\nu v}_2 &\ll (\ld/\tau)^\al \ld^{s-2} \n{u}_{X^{1/2}_\z} \n{v}_{X^{1/2}_\z},
\end{align}
and
\begin{align}
\label{geq}
\sum_{\substack{\mu \leq \nu\leq \tau/8 \\ \ld \leq \nu < \tau/8} }\n{Q_\mu u}_q \n{Q_\nu v}_2 &\ll \ld^{-1}(\ld/\tau)^\al \n{u}_{X^{1/2}_\z} \n{v}_{X^{1/2}_\z}.
\end{align}
\end{lemma}
\begin{proof}
We may interpolate~\eqref{bandstrichartz} with the trivial estimate $\n{Q_\mu u}_2 \ll (\mu \tau)^{-1/2} \n{Q_\mu u}_{X^{1/2}_\z}$ to obtain
\[\n{Q_\mu u}_q \ll (\mu/\tau)^{(1-\th)/n} (\mu \tau)^{-\th/2} \n{Q_\mu u}_{X_\z^{1/2}}\]
where $1/q = 1/2-1/p$ and $\th$ is such that $p=n/(1-\th)$. Combining this with the trivial $L^2$ estimate for $v$, we obtain
\begin{align*}
\n{Q_\mu u}_q \n{Q_\nu v}_2& \ll B_{\mu,\nu}\n{Q_\mu u}_{X^{1/2}_\z} \n{Q_\nu v}_{X^{1/2}_\z}
\end{align*}
where 
\[B_{\mu, \nu} := \tau^{-(1-\th)/n - \th/2-1/2} \mu^{(1-\th)/n-\th/2} \nu^{-1/2}.\]
Set
\[\b: = \frac{1-\th}{n} - \frac{\th}{2}.\]
Suppose first that $\b > 0$. When $\nu \geq \ld$ we have
\begin{align*}
B_{\mu,\nu}&\ll \tau^{-(1-\th)/n - \th/2-1/2} \ld^{(1-\th)/n - \th/2-1/2} (\mu/\nu)^{\b} \\
&=\ld^{-\th-1} (\ld/\tau)^{(1-\th)/n + \th/2 + 1/2} (\mu/\nu)^{\b},
\end{align*}
We take $\al \leq (1-\th)/n+(1+\th)/2$ and use the discrete Young's inequality to establish~\eqref{geq}.

Suppose now that $\nu < \ld$. When $n=3$, we set $\th = 0$,
\begin{align*}
(\nu/\ld)^{1/3}  B_{\mu,\nu} &= (\nu/\ld)^{1/3} \tau^{-5/6} \mu^{1/3} \nu^{-1/2}\\
&=(\mu/\nu)^{1/3} (\nu/\tau)^{1/6}   (\ld/\tau)^{2/3}\ld^{-1}.
\end{align*}
By Young's inequality we have~\eqref{leq} for $\al\leq 2/3$. When $n = 4$ we take $\th$ to be zero and obtain
\begin{align*}
(\nu/\ld)^{1/4} B_{\mu,\nu} & = (\nu/\ld)^{1/4} \tau^{-3/4} \mu^{1/4} \nu^{-1/2}\\
&= (\mu/\nu)^{1/4}(\ld/\tau)^{3/4} \ld^{-1}.
\end{align*}
Applying Young's inequality we have~\eqref{leq} for $\al\leq 3/4$.

When $n > 4$, we have
\begin{align*}
(\nu/\ld)^{(1-\th)/n} B_{\mu,\nu} & = (\mu/\nu)^\b \nu^{-1/2+2(1-\th)/n - \th/2} \ld^{s-2} (\ld/\tau)^{(1-\th)/n+\th/2+1/2}.
\end{align*}
In this case we have~\eqref{leq} for $\al \leq (1-\th)/n+\th/2+ 1/2$ 

In higher dimensions, we also want to consider the case $(1-\th)/n - \th/2 \leq 0$. For $\nu \geq \ld$ we have
\begin{align*}
B_{\mu,\nu} &\leq \mu^\b \ld^{-1/2} \tau^{-(1-\th)/n-\th/2-1/2}\\
&\ll \ld^{-1} \tau^{-(1-\th)/n-\th/2}.
\end{align*}
Then we have~\eqref{geq} for $\al < (1-\th)/n + \th/2$, since there are only $\sim \log \tau$ possible values of $\mu,\nu$. 

For $\ld \geq \nu$ we have
\begin{align*}
(\nu/\ld)^{(1-\th)/n} B_{\mu,\nu} & \ll \nu^{(1-\th)/n-1/2} \ld^{-2(1-\th)/n-\th/2-1/2} (\ld/\tau)^{(1-\th)/n+\th/2+1/2}.
\end{align*}
Thus we have~\eqref{leq} for $\al \leq  (1-\th)/n + \th/2 + 1/2$.
\end{proof}

Let $P_\ld$ denote the Littlewood-Paley projections, and let $P^1_\mu$ denote the Littlewood-Paley projections in the $e_1$ direction. Then
\begin{lemma}
\label{mqlemma}
Let $s,p$ be as in Theorem~\ref{maintheorem}. Then for any $f \in W^{s-1,p}(\R^n) \cap L^n(\R^n)$, 
\begin{align*}
\n{m_{\na f}}_{X^{1/2}_\z \to X^{-1/2}_\z} & \ll \n{f}_n + \sup_{\nu\leq \ld \leq 100\tau} (\ld/\tau)^{\b} (\ld/\nu)^{1/p} \ld^{s-1} \n{P_\ld P^1_{\leq 8\nu}f}_p ,
\end{align*}
where $\b>0$
\end{lemma}
\begin{proof}
Write
\[m_{\na f}(u,v) = m_{\na f} (Q_h u, Q_h v) + m_{\na f}(Q_h u, Q_l v) + m_{\na f}(Q_l u, Q_h v) + m_{\na f}(Q_l u, Q_l v).\]
We can treat all but the last term using~\eqref{hfestimate},~\eqref{strichartz}. Integrating by parts,
\begin{align*}
\a{m_{\na f}(Q_h u,Q_h v)} & \ll \n{f}_n \n{Q_h \na u}_2 \n{Q_h v}_{2n/(n-2)} + \n{f}_n \n{Q_h u}_{2n/(n-2)} \n{Q_h \na v}_2 \\
&\ll \n{f}_n \n{u}_{X^{1/2}_\z} \n{v}_{X^{1/2}_\z}.
\end{align*}
Since $Q_l v$ is supported in $\a{\xi} \ll \tau$,
\begin{align*}
\a{m_{\na f}(Q_h u, Q_l v)} & \ll \n{f}_n \n{Q_h \na u}_2 \n{Q_l v}_{2n/(n-2)} + \n{f}_n \n{Q_h u}_{2} \n{Q_l \na v}_{2n/(n-2)} \\
&\ll \n{f}_n  \n{Q_h u}_{H^1_\tau} \n{Q_l v}_{2n/(n-2)} \\
&\ll \n{f}_n \n{u}_{X^{1/2}_\z} \n{v}_{X^{1/2}_\z}.
\end{align*}

It remains to estimate $m_{\na f}(Q_l u,Q_l v)$. We have
\begin{align}
\label{qsum}
m_{\na f}(Q_l u, Q_l v) & = \sum_{\mu,\nu,\ld}\int (\na P_\ld f)\, Q_\mu u \, \bar{Q_\nu v}\,dx.
\end{align}
Suppose $\mu \leq \nu$ (the case $\mu > \nu$ is identical). Because $Q_\mu u\, \bar{Q_\nu v}$ has Fourier support in $\{\xi: \a{\xi_1} \leq 2 \nu\}$, Plancherel's theorem and H\"older's inequality give
\begin{align*}
\as{\int (\na P_\ld f)\, Q_\mu u \, \bar{Q_\nu v}\,dx} & = \as{\int P^1_{\leq 8 \nu}( \na P_\ld f)\, Q_\mu u\,\bar{Q_\nu v}\,dx}\\
& \ll \n{P^1_{\leq 8 \nu} \na P_\ld f}_p \n{Q_\mu u}_q \n{Q_\nu v}_2.
\end{align*}
Furthermore, since $Q_\mu u \, \bar{Q_\nu v}$ has Fourier support in $\{\a{\xi} \ll 100\tau\}$, we can assume $\ld \leq 100\tau$ in this sum. Applying Lemma~\ref{calculation},
we get
\begin{align*}
\a{m_{\na f} (Q_l u, Q_l v)} & \ll \sum_{\substack{\nu\geq \ld\\\mu\leq \nu}}  \n{\na P_\ld f}_p \n{Q_\mu u}_q \n{Q_\nu v}_2 \\
&\qq+ \sum_{\substack{\nu < \ld\leq 100\tau\\\mu\leq\nu}} (\ld/\nu)^{1/p} (\nu/\ld)^{1/p} \n{\na P_{\ld} P^1_{\leq 8\nu}f}_p \n{Q_\mu u}_q \n{Q_\nu v}_2\\
&\ll \sum_{\ld\leq 100\tau} \{(\ld/\tau)^\al \ld^{-1} \n{\na P_\ld f}_p \\
&\qq\qq+ \sup_{\nu\leq \ld} (\ld/\tau)^{\al} (\ld/\nu)^{1/p} \ld^{s-2} \n{\na P_\ld P^1_{\leq 8\nu} f}_p\}\\
&\qq\times\n{u}_{X^{1/2}_\z} \n{v}_{X^{1/2}_\z}\\
&\ll (\n{f}_p +  \sup_{\nu\leq \ld\leq 100\tau}(\ld/\tau)^{\al/2} (\ld/\nu)^{1/p} \ld^{s-1} \n{P_\ld P^1_{\leq 8\nu} f}_p)\\
&\qq\times\n{u}_{X^{1/2}_\z} \n{v}_{X^{1/2}_\z}.
\end{align*}
\end{proof}

\section{Averaging}
Given any vector $\om\in S^{n-1}$, we define $P^\om_{\mu}$ to be Littlewood-Paley projection in $\om$ direction. Let $\mu$ denote Haar measure on $O(n)$, normalized so that if $\si$ is the usual spherical measure on $S^{n-1}$ and $f:S^{n-1} \to \R$ is integrable, then for any $\th\in S^{n-1}$ we have
\begin{equation}
\label{haarformula}
\int_{O(n)} f(U\th)\,d\mu(U) = \int_{S^{n-1}} f(\om) \,d\si(\om).
\end{equation}

\begin{lemma}
\label{planeavg}
Suppose $p \in [2,\infty]$. Let $f \in L^p(\R^n)$. For $U \in O(n)$ and $\nu \leq \ld$, define
\[A_{\ld,\nu}(U) = (\ld/\nu)^{1/p}\n{P_\ld P_{\leq \nu}^{Ue_1} f}_p.\]
Then 
\[\n{A_{\ld,\nu}}_{L^p(O(n))} \ll \n{f}_p\]
\end{lemma}
\begin{proof}
We define an operator $T$ mapping functions on $\R^n$ to functions on $O(n) \times \R^n$ by
\[Tf(U,x) = P_\ld P_{\leq \nu}^{Ue_1} f(x).\]
The lemma asserts that this operator is bounded from $L^p(\R^n)$ to $L^p(O(n)\times \R^n)$. By interpolation, it suffices to establish this at the endpoints $p = 2$ and $p=\infty$. 

When $p = \infty$ this is just the fact that the Littlewood-Paley projections are bounded on $L^\infty$.

When $p=2$ we use Plancherel's theorem and Fubini. 
\begin{align*}
\n{Tf}_{L^2}^2 & \sim \int_{O(n)} \int_{\R^n} \a{\phi(\xi/\ld) \chi(\xi\cd (Ue_1)/\nu) \hat f(\xi) }^2\,d\xi \,d\mu(U)\\
&\leq \sps{ \sup_\xi \int_{O(n)} \a{\phi(\xi/\ld) \chi(\xi\cd (Ue_1)/\nu)}^2\,d\mu(U)} \n{f}_2^2.
\end{align*}
Here $\phi$ is supported on an annulus, and $\chi$ is supported on an interval. We estimate the last integral using~\eqref{haarformula} and spherical coordinates:
\begin{align*}
\int_{O(n)} \a{\phi(\xi/\ld) \chi(\xi\cd (Ue_1)/\nu)}^2 \,d\mu(U) &\ll \sup_{\a{\xi} \sim \ld} \int_{S^{n-1}} \a{\chi(\a{\xi} \om\cd e_1/\nu)}^2 \,d\si(\om)\\
&\ll\sup_{\a{\xi} \sim \ld}  \int_0^\pi \a{\chi(\a{\xi} \cos \th/\nu)}^2 \, \sin(\th)^{n-2}\,d\th\\
&\ll \sup_{\a{\xi} \sim \ld} \int_{-1}^1 \chi(\a{\xi} u/\nu) \,du\\
&\ll\sup_{\a{\xi} \sim \ld}  \frac{\nu}{\a \xi} \\
&\ll \frac{\nu}{\ld}.
\end{align*}
This shows that
\[\n{Tf}_2  \ll (\nu/\ld)^{1/2} \n{f}_2,\]
which completes the proof.
\end{proof}
Define $\z(\tau,U) = \tau U(e_1-ie_2)$. Our next lemma establishes that $\n{q}_{X^{-1/2}_{\z(\tau,U)}}$ is small on average. This is implied by~\cite[Lemma 3.1]{Haberman2013}, but we give a simpler proof here, based on~\cite{Nguyen2014}:
\begin{lemma}
\label{qavg}
If $f \in \dot H^{-1}$, then
\begin{align*}
M^{-1} \int_{M}^{2M} \int_{O(n)} \n{f}_{\dot X^{-1/2}_{\z(\tau,U)}}^2 \, d\mu(U) \,d\tau &\ll \n{P_{\geq 100M} f}_{\dot H^{-1}}^2 + M^{-1} \n{P_{< 100 M}f}_{\dot H^{-1/2}}^2.
\end{align*}
\end{lemma}
\begin{proof}
This is true if $f$ is supported at frequencies $\a{\xi} \geq 100M$, because there we have $\a{p_\z(\xi)} \geq \a \xi^2$. Thus we may assume that $f$ is supported at frequencies $\a{\xi} \ll M$, where we have $\a{p_\z(\xi)} \gg 2\tau\a{\xi\cd (U e_1)} + \a{-\a \xi^2 + 2\tau \xi \cd (Ue_2)}$. Here we use Plancherel and the identity $U^T = U^{-1}$ and estimate as in Lemma~\ref{planeavg} by
\begin{align*}
\n{\a{\na}^{-1/2}f}_2^2  \sup_{\a{\xi} \leq 100M} \frac{\a{\xi}}{M} \int_M^{2M}  \int_{O(n)} (2\tau \a {(U^{-1} \xi) \cd e_1} + \a{-\a{\xi}^2 + 2\tau (U^{-1} \xi)\cd e_2})^{-1} \,d\mu(U)\,d\tau.
\end{align*}
By~\eqref{haarformula}, the quantity inside the supremum is given by
\begin{align*}
\frac{1}{M} \int_M^{2M} \int_{S^{n-1}} (2\tau\a{\om \cd e_1} + \a{-\a\xi + 2\tau\om\cd e_2})^{-1} \,d\si(\om)\,d\tau.
\end{align*}
We view $(\tau,\om)$ as polar coordinates and change variables to $u = \tau\om$. Then in the region $\tau \in [M,2M]$ the volume element $du$ is bounded below by $M^{n-1} \, d\si(\om) \,d\tau$, so this integral is bounded by
\[\frac{1}{M^n} \int_{\a u \in [M,2M]} (2\a{u_1} + \a{-\a{\xi} + 2u_2})^{-1}\,du.\]
Writing $v = (u_1,u_2)$, and integrating over the remaining variables, we bound by
\begin{align*}
\frac{1 }{M^n} M^{n-2} \int_{B(0,2M)} (2\a {v_1} + \a{-\a\xi + 2v_2})^{-1} \,dv&\leq \frac{1}{M^2} \int_{B(0,2M)} \a{v}^{-1} \,dv\\
&\sim \frac{1}{M}.
\end{align*}
\end{proof}

We summarize our estimates so far in the following
\begin{theorem}
\label{estimatetheorem}
Let $s,p$ be as in Theorem~\ref{maintheorem}, and let $\g$ be a positive real-valued function on $\R^n$ such that $\na \log \g \in W^{s-1,p}$ and  $\g = 1$ outside of a large ball $B$. For $q = \g^{-1/2}\lp \g^{1/2}$, we have
\begin{align}
\label{mqavgdecay}
M^{-1} \int_{M/2}^{2M} \int_{O(n)} \n{q}_{X^{-1/2}_{\z(\tau,U)}}^2 \,d\mu(U)\,d\tau  &\to 0.
\end{align}
Furthermore,
\begin{align}
\sup_{\tau\in[M/2,2M]}\n{m_q}_{X^{1/2}_{\z(\tau,U)} \to X^{-1/2}_{\z(\tau,U)}} & \leq C_M + A_M(U),
\label{qavgdecay}
\end{align}
where $C_M \to 0$ as $M \to \infty$ and
\begin{equation}
\label{liminf}
\sum_{k > 2} k^{-1} \n{A_{2^k}}_{L^p(O(n))}^p <\infty.
\end{equation}
\end{theorem}
\begin{proof}
First, we write
\[\g^{-1/2} \lp \g^{1/2} = \f 1 2 \lp \log \g + \f 1 4 \a{\na \log \g}^2 = \sum_i \na_i f_i + h,\]
where $f_i \in W^{s-1,p}$ and $h \in L^{p/2}$.

We decompose each term into a good part and a bad part. Let $\phi_\epsilon = \epsilon^{-n} \phi(x/\epsilon)$, where $\phi$ is a $C_0^\infty$ function supported on the unit ball and $\int \phi = 1$. Define $f_\epsilon = f * \phi_\epsilon$. 

By~\eqref{linfinity}, we have
\begin{align*}
\n{m_{\na f_\e}}_{X^{1/2}_\z \to X^{-1/2}_\z} + \n{m_{h_\e}}_{X^{1/2}_\z \to X^{-1/2}_\z} & \ll \tau^{-1}( \n{\na f_\e}_\infty + \n{h_\e}_\infty)\\
&\ll \tau^{-1} \e^{-2}( \n{f}_n + \n{h}_{n/2}).
\end{align*}
We also have
\begin{align*}
\n{\na f_\e}_{X^{-1/2}_\z} & \ll \tau^{-1/2} \n{\na f_\e}_2 \\
&\ll \tau^{-1/2} \e^{-1} \n{f}_2 \\
&\ll \tau^{-1/2} \e^{-1} \n{f}_n,
\end{align*}
since $n > 2$ and $f$ is compactly supported. For $n\geq 4$ we have
\begin{equation}
\begin{aligned}
\n{h}_{X^{-1/2}_\z} &\ll \tau^{-1/2} \n{h}_2\\
&\ll \tau^{-1/2} \n{h}_{n/2},
\label{highhestimate}
\end{aligned}
\end{equation}
and for $n = 3$ we have
\begin{align*}
\n{h_\e}_{X^{-1/2}_\z} & \ll \tau^{-1/2} \n{h_\e}_2 \\
&\ll \tau^{-1/2} \e^{-1/2} \n{h}_{3/2}.
\end{align*}
Taking $\e = M^{-1/4}$, we find that if we replace $q$ with $q_\e$ then the left hand sides of~\eqref{mqavgdecay} and~\eqref{qavgdecay} vanish as $\tau \to \infty$.

It remains to treat the bad part $q-q_\e$. Let $g = f - f_\e$, and define
\[A(\tau,U) = \n{m_{\na g}}_{X^{1/2}_{\z(\tau,U)}\to X^{-1/2}_{\z(\tau,U)}}.\]
Using Lemma~\ref{mqlemma}, we have
\[\sup_{\tau\in [M/2,2M]} A(\tau,U) \ll \n{g}_{L^p} + \sps{\sum_{1\leq \nu\leq \ld \leq 4M} [(\ld/M)^{\b} \ld^{s-1} A_{\ld,\nu}(U)]^p}^{1/p}, \]
where $A_{\ld,\nu}(U) = (\ld/\nu)^{1/p} \n{P_\ld P_{\leq 8\nu}^{Ue_1} g}_{L^p}$. As $M \to \infty$, we have $\e = M^{-1/4} \to 0$, so $\n{g}_{L^p} \to 0$. We take $A_M(U)$ to be the second term on the right hand side of this inequality, which is clearly a measurable function on $O(n)$. Now, $P_\ld g = P_\ld P_{\sim \ld} g$, where $P_{\sim \ld} g = \sum_{\ld/16 \leq \mu\leq 16\ld} P_\mu g$. Applying Lemma~\ref{planeavg}, we have
\begin{align*}
\n{A_M(U)}_{L^p(O(n))}^p &\ll \sum_{1\leq \nu \leq \ld \leq M/4} [(\ld/M)^{\b} \ld^{s-1} \n{P_{\sim \ld} g}_{L^p}]^p\\
&\ll \log M \sum_{1\leq \ld \leq M/4} [(\ld/M)^\b \ld^{s-1} \n{P_{\sim \ld} f}_{L^p}]^p.
\end{align*}
We control this quantity by taking a weighted sum over dyadic integers $M$, as in~\cite{Nguyen2014}. Namely, we have
\begin{align*}
\sum_{M\geq 2} (\log M)^{-1} \n{A_{M}(U)}^p_{L^p(O(n))} & \ll \sum_\ld \sum_{M \geq 4\ld} (\ld/M)^{\b p} [\ld^{s-1} \n{P_{\sim \ld} f}_{L^p}]^p\\
&\ll \sum_\ld [\ld^{s-1} \n{P_{\sim \ld} f}_{L^p}]^p.
\end{align*}
The last term is controlled by $\n{f}_{W^{s-1,p}}$ as a consequence of the Littlewood-Paley square function estimate. Thus we obtain~\eqref{liminf}.

By Lemma~\ref{qavg}, we have
\[M^{-1}\int_{M/2}^{2M} \int_{O(n)} \n{\na g}^2\,d\mu(U)\,d\tau \ll \n{g}_{L^2}^2\ll \n{g}_{L^p}^2\to 0.\]

Next we treat $h-h_\e$. When $n \geq 4$ we have $\n{h-h_\e}_{X^{-1/2}_\z} \to 0$ by~\eqref{highhestimate}. When $n=3$, we have
\begin{align*}
\n{h-h_\e}_{X^{-1/2}_\z} & \ll \n{h-h_\e}_{H^{-1/2}}\\
&\ll \n{h-h_\e}_{3/2}\\
&\to 0
\end{align*}
by Sobolev embedding. Finally, by~\eqref{lowerorder} we have
\[\n{m_{h-h_\e}}_{X^{1/2}_\z\to X^{-1/2}_\z} \ll \n{h-h_\e}_{n/2} \to 0.\]
\end{proof}

\section{Localization}
Because our problem is localized to a compact set, the uncertainty principle implies that the $X^{1/2}_\z$ norm is equivalent to the $\dot X^{1/2}_\z$ norm. To make this precise, we state the following
\begin{lemma}[\cite{Haberman2013}]
Let $\phi$ be a fixed Schwartz function. Then
\begin{align}
\label{dualloc}
\n{\phi u}_{\dot X^{-1/2}_\z} &\ll_\phi \n{u}_{X_\z^{-1/2}}\\
\label{loc}
\n{\phi u}_{X^{1/2}_\z} &\ll_\phi \n{u}_{\dot X^{1/2}_\z},
\end{align}
where the constants depend on the seminorms $\n{x^\al \na^\b \phi}_\infty$.
\end{lemma}
In particular, we have
\begin{lemma}
Suppose that $q$ is compactly supported. Then
\begin{equation}
\n{m_q}_{\dot X_\z^{1/2}\to \dot X_\z^{-1/2}} \ll \n{m_q}_{X_\z^{1/2}\to X_\z^{-1/2}}.
\label{compactbilinear}
\end{equation}
\end{lemma}
\begin{proof}
Let $\phi$ be a Schwartz function that it equal to one on the support of $q$. Then
\begin{align*}
\a{\j{m_q u, v}} & = \a{\j{m_q \phi u, \phi v}} \\
&\ll \n{m_q}_{X^{1/2}_\z \to X^{-1/2}_\z} \n{\phi u}_{X^{1/2}_\z}\n{\phi v}_{X^{1/2}_\z}\\
&\ll \n{m_q}_{X^{1/2}_\z \to X^{-1/2}_\z} \n{u}_{\dot X^{1/2}_\z} \n{v}_{\dot X^{1/2}_\z}
\end{align*}
\end{proof}

We record the following useful fact:
\begin{lemma}
\label{fudgezeta}
Suppose $\z, \ti \z\in \C^n$ satisfy $\z \cd \z = \ti \z \cd \ti \z = 0$. Then
\[\n{u}_{X^{b}_\z} \ll (1+\a{\z-\ti \z})^{\a b} \n{u}_{X^b_{\ti \z}}.\]
\end{lemma}
\begin{proof}
We have
\begin{align*}
\a{p_\z}& \leq \a{p_{\ti \z}} + 2 \a{(\z-\ti \z) \cd \xi} \\
&\leq \a{p_{\ti \z}} + 2 \a{\z-\ti \z} \a{\xi}\\
&\ll ( 1 + \a{\z-\ti \z}) (\a{p_{\ti \z}} + \tau)
\end{align*}
by~\eqref{pbehavior}. 
\end{proof}

\section{Proof of the main theorem}
We summarize some known results which allow us to extend the $\g_i$ to all of $\R^n$. First we transfer the problem to the interior, as in~\cite{Sylvester1987}.
\begin{lemma}
\label{weakequality}
Suppose $n \geq 3$. Let $\g_1,\g_2 \in W^{1,n}(\R^n)$ be functions such that $0<c \leq \g_i \leq c^{-1}$ for some $c$. If $\g_1 = \g_2$ outside $\Om$ and $\Ld_{\g_1}=\Ld_{\g_2}$, then for $q_j = \lp \sqrt {\g_j}/\sqrt{\g_j}$, we have
\[\j{q_1,v_1 v_2} = \j{q_2,v_1 v_2}\]
when each $v_j$ is a solution in $H^1_\loc(\R^n)$ to $\lp v_j - q_j v_j = 0$.
\end{lemma}
\begin{proof}
See~\cite{Brown2003}. 
\end{proof}

The following argument is apparently due to Alessandrini. It amounts to the fact that $q_1 = q_2$ implies that the function $\log \g_1 - \log \g_2$ solves the Dirichlet problem $\div \sqrt{g_1 g_2}\na u = 0$ with $u = 0$ at infinity. See~\cite{Sylvester1987,Brown1996,Brown2003}.
\begin{lemma}
\label{gaidentity}
Let $\g_i, q_i$ be as in Lemma~\ref{weakequality}, and suppose that $q_1 = q_2$ in the sense of distributions. Then $\g_1 = \g_2$.
\end{lemma}
\begin{proof}
First, have $q_i \in H^{-1}(\R^n)$ for each $i$. To see this we note that
\begin{align*}
\n{q}_{H^{-1}} & = \n{\f 1 2 \lp \log \g + \f 1 2  \a{\na \log \g}^2}_{H^{-1}}\\
&\ll \n{\na \log \g}_2 + \n{\a{\na \log \g}^2}_{H^{-1}}\\
&\ll \n{\na \log \g}_n + \n{\na \log \g}_{4n/(n+2)}^2 \\
&\ll \n{\na \log \g}_n + \n{\na \log \g}_n
\end{align*}
by Sobolev embedding and H\"older's inequality. It follows that we may test $q_1 - q_2$ against the function $g_1 g_2( \log g_1 - \log g_2) \in H^1(\R^n)$, where $g_i = \sqrt{\g_i}$. This gives
\begin{align*}
0& = \int [\na g_1 \cd \na (g_2 (\log g_1 - \log g_2)) - \na g_2 \cd \na(g_1 (\log g_1 - \log g_2))]\,dx\\
&=\int (g_2 \na g_1 - g_1 \na g_2) \cd \na(\log g_1 - \log g_2)\,dx\\
&=\int g_1 g_2 \a{\na(\log g_1 - \log g_2)}^2 \,dx,
\end{align*}
which implies that $g_1 = g_2$. 
\end{proof}

Now we apply the boundary determination result of~\cite{Brown2013}, which implies
\begin{theorem}
\label{boundarydetermination}
Suppose that $0<c<\g_i <c^{-1}$. If $\g_i \in W^{1,1}(\Om)$ and $\Ld_{\g_1} = \Ld_{\g_2}$, then $\g_1 = \g_2$ on $\pd \Om$.
\end{theorem}

\begin{proof}[Proof of Theorem~\ref{maintheorem}]
By Theorem~\ref{boundarydetermination}, we have $\g_1 = \g_2$ on $\pd \Om$. Our assumptions imply that $s -1/p \leq 1$. Thus by~\cite{Marschall1987} we may extend the $\g_i$ to functions in $W^{s, p}$ such that $\g_1 = \g_2$ outside of $\Om$. By Lemma~\ref{weakequality}, this implies that
\begin{equation}
\label{weakglobal}
\j{q_1, v_1 v_2}  = \j{q_2, v_1 v_2}
\end{equation}
when each $v_j$ is a solution in $H^1_\loc(\R^n)$ to $\lp v_j - q_j v_j = 0$. 

Fix $r> 0$ and three orthonormal vectors $\{e_1,e_2, e_3\}$, and define 
\begin{align*}
\z_1(\tau,U) & = \tau U(e_1 - i e_2)\\
\z_2(\tau, U) & = -\z_1(\tau, U)\\
\ti \z_1 (\tau,U) &:= \tau Ue_1 + i (rUe_3 - \sqrt{\tau^2 - r^2}Ue_2)\\
\ti \z_2 (\tau,U) &:= -\tau U e_1 + i(rUe_3 + \sqrt{\tau^2 - r^2}Ue_2) 
\end{align*}
In what follows, all of inequalities will implicitly depend on $r$. For example, we have $\a{\z_i - \ti \z_i} \ll 1$. In particular, by Lemma~\ref{fudgezeta}, the spaces $X^b_{\z_i}$ and $X^b_{\ti \z_i}$ have equivalent norms.

Now let
\[F(\tau,U) = \sum_i \n{m_{q_i}}_{X^{1/2}_{\z_i(\tau,U)} \to X^{-1/2}_{\z_i(\tau,U)}}^p + \sum_{i,j} \n{q_i}_{X^{-1/2}_{\z_j(\tau, U)}}^2.\]
By Theorem~\ref{estimatetheorem} and the fact that $\sum_{k>1} k^{-1} = \infty$, we have
\begin{equation}
\label{globalestimate}
\liminf_{M\to\infty} M^{-1} \int_{M}^{2M}\int_{O(n)} F(\tau,U) \,d\mu(U)\,d\tau = 0.
\end{equation}
Now we use an argument from~\cite[Section 3.3]{Nguyen2014} to select $\tau$ and $U$.
Their first observation is that if $\e > 0$ and $B_\e = \{U \in O(n): \n{U - I} < \e\}$, then by simply restricting~\eqref{globalestimate} we have
\[\liminf_{M\to\infty} M^{-1} \mu(B_\e)^{-1} \int_{M}^{2M}\int_{B_\e} F(\tau,U) \,d\mu(U)\,d\tau =0.\]
Thus we may choose, for a sequence of $M=M_l$ such that $M_l \to \infty$, some $\tau = \tau_{\e,l} \in [M_l, 2M_l]$, $U = U_{\e,l} \in B_\e$ and $\d=\d_{\e,l} > 0$ such that 
\begin{equation}
\label{smallness}
\sum_i \n{m_{q_i}}_{X^{1/2}_{\z_i(\tau,U)} \to X^{-1/2}_{\z_i(\tau,U)}} + \sum_{i,j} \n{q_i}_{X^{-1/2}_{\z_j(\tau, U)}} \leq \d
\end{equation}
where $\d_{\e,l} \to 0$ as $l\to \infty$. 

By~\eqref{compactbilinear}, we have
\[\n{m_{q_i}}_{\dot X^{1/2}_{\ti\z_i(\tau,U)} \to \dot X^{-1/2}_{\ti \z_i(\tau,U)}} \ll \n{m_{q_i}}_{X^{1/2}_{\ti\z_i(\tau,U)} \to X^{-1/2}_{\ti \z_i(\tau,U)}}\]
It follows that,  
\[\n{m_{q_i}}_{\dot X^{1/2}_{\ti\z_i(\tau,U)} \to \dot X^{-1/2}_{\ti \z_i(\tau,U)}} \ll \d_{\e,l}\]
Since $\d_{\e,l} \to 0$ as $l \to \infty$, we can choose $l$ large enough that the left hand side is less than $1/2$. Since $\n{\lp_\z^{-1}}_{\dot X^{-1/2}_\z \to \dot X^{1/2}_\z} = 1$ for any $\z$, we can use the contraction mapping principle to construct solutions $\psi_i \in \dot X^{1/2}_{\ti \z_i(\tau,U)}$ to the equations $(\lp_{\ti\z_i(\tau,U)} - m_{q_i})\psi_i = q_i$, satisfying
\[\n{\psi_i}_{\dot X^{1/2}_{\ti \z_i(\tau,U)}} \ll \n{q}_{\dot X^{-1/2}_{\ti \z_i(\tau,U)}}.\]
Note that by~\eqref{easyenergy}, such a solution lies in $H^1_\loc(\R^n)$. This implies that the corresponding solution $v_i = e^{x\cd \ti\z_i(\tau,U)}(1+\psi_i)$ to the Schr\"odinger equation $(\lp-q_i)v_i$ lies in $H^1_\loc(\R^n)$ as well. 

Let $k =2r U e_3$. By~\eqref{weakglobal},
\begin{align*}
0&=\j{q_1 - q_2, e^{ik\cd x} (1 + \psi_1)(1 + \psi_2)} \\
&=\j{q_1 - q_2, e^{ik\cd x}} + \j{q_1-q_2, e^{ik\cd x} \psi_1 \psi_2} + \j{q_1-q_2, e^{ik\cd x} (\psi_1 + \psi_2)}.
\end{align*}
We need to show that the second and third terms are small. Let $\phi$ be a Schwartz function that is equal to one on the support of $q$. Then
\begin{align*}
\a{\j{q_1, e^{ik\cd x}\psi_1 \psi_2}} & =\a{\j{m_{q_1} e^{-ik\cd x} \bar \psi_2, \psi_1}} \\
&\ll \n{e^{-ik\cd x} \phi \bar \psi_2}_{ X^{1/2}_{\z_1(\tau,U)}} \n{\phi \psi_1}_{X^{1/2}_{\z_1(\tau,U)}}\\
&= \n{e^{ik\cd x} \phi \psi_2}_{X^{1/2}_{\z_2(\tau,U)}} \n{\phi \psi_1}_{X^{1/2}_{\z_1(\tau,U)}}\\
&\ll \n{\psi_2}_{\dot X^{1/2}_{\ti \z_2(\tau,U)}} \n{\psi_1}_{\dot X^{1/2}_{\ti \z_1(\tau,U)}}\\
&\ll \n{q_2}_{ X^{-1/2}_{\z_2(\tau, U)}}\n{q_1}_{X^{-1/2}_{\z_1(\tau, U)}},
\end{align*}
since the seminorms of $e^{-ik\cd x} \phi$ are bounded with a bound depending only on $r$. We can bound the $q_2$ term in the same way. On the other hand, we have
\begin{align*}
\a{\j{q_i, e^{ik\cd x} \psi_1}} & \ll \n{q_i}_{X^{-1/2}_{\z_1(\tau,U)}} \n{\psi_1}_{\dot X^{1/2}_{\ti \z_1(\tau,U)} }\\
&\ll \n{q_i}_{X_{\z_1(\tau,U)}^{-1/2}} \n{q_1}_{X^{-1/2}_{\z_1(\tau,U)}}
\end{align*}
by duality of $\dot X^{1/2}_{\z_1(\tau,U)}$ and $\dot X^{-1/2}_{\z_1(\tau,U)}$. The terms with $\psi_2$ are the same. In summary, we obtain
\begin{align}
\label{tozero}
\a{(\hat q_1 - \hat q_2)(2rUe_3)}& \ll \sum_{1\leq i,j,k,l\leq 2} \n{q_i}_{X^{-1/2}_{\z_j(\tau,U)}} \n{q_k}_{X^{-1/2}_{\z_l(\tau,U)}} \ll \d^2
\end{align}
by~\eqref{smallness}. 

To finish the proof, we again follow~\cite{Nguyen2014}. Since $B_\e$ is compact, we may pass to a subsequence such that $U_{\e,l} \to U_\e$ for some $U_\e \in B_\e$. Since the $\hat q_i$ are continuous, we may pass to the limit in~\eqref{tozero} to obtain 
\[\a{\hat q_1 - \hat q_2}(2r U_\e e_3) \ll \lim_{l\to \infty} \d_{\e,l}^2 = 0.\]
Note that by construction, we have $U_\e \to I$ as $\e \to 0$. Thus, by taking limits again, we obtain $(\hat q_1 - \hat q_2)(2 r e_3) = 0$. Since $e_3 \in S^{n-1}$ and $r$ were arbitrary, this means that $\hat q_1 - \hat q_2=0$.
\end{proof}
\section*{Acknowledgments}
The author would like to thank his advisor, Daniel Tataru, for his patient guidance and encouragement. He would also like to thank Gunther Uhlmann, Russell Brown, Alberto Ruiz and Mikko Salo for many helpful conversations. Finally, the author would like to thank the anonymous referee for carefully reading the manuscript and suggesting many corrections and improvements.
\bibliography{/home/boaz/Documents/library.bib}{}
\bibliographystyle{amsalpha}
\end{document}